\documentclass[11pt,a4paper]{amsart}
\usepackage{amsmath, amssymb, amsfonts, amsthm, float, stmaryrd, epsfig}
\usepackage[abbrev,alphabetic]{amsrefs}
\usepackage[pdfborder={0 0 0 [0 0 ]},bookmarksdepth=3, bookmarksopen=true, unicode]{hyperref}
\usepackage[latin1]{inputenc}
\usepackage[capitalize]{cleveref}
\usepackage{color}
\usepackage[ps,all,arc,rotate]{xy}
\usepackage{bbm} 

\usepackage{stackengine,scalerel}
\stackMath

\hypersetup{
    colorlinks=true,
    linkcolor=black,
    citecolor=black,
    filecolor=black,
    urlcolor=black,
}

\numberwithin{figure}{section}
\numberwithin{table}{section}

\theoremstyle{plain}
\newtheorem{thm}{Theorem}[section]
\crefname{thm}{Theorem}{Theorems}
\newtheorem*{prop*}{Proposition}
\newtheorem*{thm*}{Theorem}
\newtheorem{prop}[thm]{Proposition}
\crefname{prop}{Proposition}{Propositions}
\newtheorem{lem}[thm]{Lemma}
\crefname{lem}{Lemma}{Lemmata}
\newtheorem{cor}[thm]{Corollary}
\crefname{cor}{Corollary}{Corollaries}

\newtheorem{conj}[thm]{Conjecture}
\crefname{conj}{Conjecture}{Conjectures}
\crefname{equation}{Equation}{Equations}

\newtheorem*{BaersConj}{Baer's Noetherian Conjecture}

\theoremstyle{definition}

\newtheorem{dfn}[thm]{Definition}
\newtheorem*{dfn*}{Definition}

\theoremstyle{remark}

\newtheoremstyle{maintheorem}{}{}{\itshape}{}{\bfseries}{}{.5em}{#1 \!\thmnote{\ #3}.}
\theoremstyle{maintheorem}

    \newtheoremstyle{TheoremNum}
        {\topsep}{\topsep} 
        {\itshape} 
        {-0.25cm} 
        {\bfseries} 
        {.} 
        { }  
        {\thmname{#1}\thmnote{ \bfseries #3}}
    \theoremstyle{TheoremNum}
    \newtheorem{duplicate}{}

\makeatletter
\let\c@figure\c@thm
\let\c@table\c@thm
\makeatother
\crefname{figure}{Figure}{Figures}
\crefname{table}{Table}{Tables}

\newcommand{\typeF}[1]{\mathtt{F}_{#1}}
\newcommand{\typeFP}[1]{\mathtt{FP}_{#1}}
\newcommand{\typeVF}{\mathtt{VF}}
\newcommand{\typeVFP}{\mathtt{VFP}}

\def\R{\mathbb{R}}
\def\N{\mathbb{N}}
\def\Z{\mathbb{Z}}

\def\1{\mathbbm{1}}

\def\s-{\smallsetminus}

\def\iff{if and only if }

\binoppenalty=\maxdimen
\relpenalty=\maxdimen

\usepackage[foot]{amsaddr}
\author{Sam Hughes}
\email{sam.hughes@maths.ox.ac.uk}
\author{Dawid Kielak}
\email{kielak@maths.ox.ac.uk}
\address[S.~Hughes and D.~Kielak]{Mathematical Institute,
	University of Oxford,
	Andrew Wiles Building,
	Radcliffe Observatory Quarter,
	Woodstock Road,
	Oxford
	OX2 6GG,
	United Kingdom}
\author{Peter H. Kropholler}
\email{p.h.kropholler@soton.ac.uk}
\author{Ian J. Leary}
\email{i.j.leary@soton.ac.uk}
\address[P.H.~Kropholler and I.J.~Leary]{School of Mathematical Sciences,
	University of Southampton,
	Southampton
	SO17 1BJ
	United Kingdom}

\title{Coherence for elementary amenable groups}

\begin{document}

\begin{abstract}
  We prove that for an elementary amenable group, coherence of the
  group, homological coherence of the group, and coherence of the
  integral group ring are all equivalent.  This generalises a result
  of Bieri and Strebel for finitely generated soluble groups.
\end{abstract}

\maketitle

\section{Introduction}
A group $G$ is \emph{type $\typeF{n}$} if $G$ admits a model for a $\mathsf{K}(G,1)$ with finite $n$-skeleton, \emph{type $\typeF{\infty}$} if $G$ is type $\typeF{n}$ for all $n$, and type $\typeF{}$ if $G$ admits a finite model for a $\mathsf{K}(G,1)$.  If $G$ has a finite index subgroup of type $\typeF{}$, then we say that $G$ is \emph{type $\typeVF$}.  Replacing `a model for a $\mathsf{K}(G,1)$' in the previous definitions with `a projective resolution of $\Z$ over $\Z G$', one obtains the properties $\typeFP{n}$, $\typeFP{\infty}$, $\typeFP{}$, and $\typeVFP$, respectively.

It is known from the work of Bestvina and Brady \cite{BestvinaBrady1997} that the properties of a group being finitely presentable and being of type $\typeFP{2}$ are not equivalent.  In this note we investigate the equivalence of local versions of these properties.

A group $G$ is \emph{coherent} if every finitely generated subgroup is finitely presented and \emph{homologically coherent} if every finitely generated subgroup is of type $\typeFP{2}$.  A ring $R$ is \emph{coherent} if every finitely generated left ideal is finitely presented. These various notions of coherence are being very actively investigated -- let us mention here the remarkable recent article of Jaikin-Zapirain--Linton \cite{JaikinZapirain2023} that establishes coherence of one-relator groups via the homological variants.

If  $\Z G$ is coherent or $G$ is coherent, then $G$ is homologically coherent. Neither converse has been established and in particular there is no known implication between the properties of $\Z G$ being coherent and $G$ being coherent.

In \cite{BieriStrebel1979}, Bieri and Strebel prove that the properties of coherence of the group, homological coherence, and coherence of the group ring are equivalent for a finitely generated soluble group. In doing so they relate the properties to ascending HNN extensions, these being groups which are generated by a \emph{vertex group} $B$ and \emph{stable letter} $t$ such that $t^{-1}Bt\supseteq B$, corresponding to actions on trees with one orbit of edges and a marked end. The HNN extension is \emph{properly ascending} if $t^{-1}Bt\supset B$.

Note that Groves \cite{Groves1978} independently established the equivalence of (2) and (3) below.

\begin{thm}[Bieri--Strebel]\label{Bieri Strebel soluble coherence}
For a finitely generated soluble group $G$ the following are equivalent:
\begin{enumerate}
	\item $G$ is coherent;
	\item $G$ is homologically coherent;
	\item $\Z G$ is coherent;
	\item $G$ is polycyclic or $G$ is a properly ascending HNN extension with polycyclic vertex group.
\end{enumerate}
\end{thm}

Our main theorem (\cref{thm main} below) concerns elementary amenable groups.
Recall that the class of \emph{elementary amenable groups} is the smallest class of groups containing all finite and abelian groups that is closed under extensions and direct limits (and isomorphisms). In this paper, classes of groups are implicitly assumed to be closed under isomorphism. The class of elementary amenable groups is embryonically present in von Neumann's seminal work \cite{Neumann1929}, then reintroduced by the short name of \emph{elementary groups} in \cite{Day1957} and studied further in \cite{Chou1980}. It can readily be shown that the class is subgroup closed and quotient closed, and that it contains all soluble groups.
Thus the following generalises Bieri and Strebel's result.

\begin{thm}\label{thm main}
For a finitely generated elementary amenable group $G$ the following are equivalent:
\begin{enumerate}
	\item $G$ is coherent;
	\item $G$ is homologically coherent;
	\item $\Z G$ is coherent;
	\item either $G$ is virtually polycyclic or $G$ is a properly ascending HNN extension with virtually polycyclic vertex group.
\end{enumerate}
\end{thm}

We combine the above with the following three observations:  firstly, that ascending HNN extensions of soluble groups are soluble (see \cref{union of sols});  secondly, that an ascending HNN extension of a group of type $\typeF{\infty}$ is also of type $\typeF{\infty}$; and thirdly, that by \cite{KrophollerMartinezPerezNucinkis2009}*{Theorem~1.1} an elementary amenable group of type $\typeF{\infty}$ is type $\typeVF$. In this way we obtain the following.

\begin{cor}
	Every (homologically) coherent elementary amenable group is virtually soluble and of type $\typeVF$.
\end{cor}

Although the conclusions of \cref{Bieri Strebel soluble coherence,thm main} are very similar, the methods of proof are quite different.

Bieri--Strebel prove that any homologically coherent soluble group $G$ has finite Pr\"ufer rank --- namely, there exists an $r$ such that every subgroup can be generated by $r$ elements.  From here they apply results of Mal'cev and the theory of nilpotent groups to deduce that $G$ is polycyclic or that $G$ is a properly ascending HNN extension with polycyclic vertex group.

In comparison, we introduce the notion of $G$ being \emph{indicably coherent} (\Cref{defn indicable}) and prove that this property is inherited by quotients of groups with no non-abelian free subgroups, and so in particular by elementary amenable groups.  We prove a general virtual splitting principle for elementary amenable groups as ascending HNN extensions of lower complexity groups.  Using these two tools we show that every finitely generated elementary amenable indicably coherent group $G$ is virtually soluble.

Since the technique we use is very much based on understanding the behaviour of epimorphisms to $\Z$ and their kernels, and since we utilise it explicitly in the proof of \cref{peters adjustment}, let us briefly discuss the Bieri--Neumann--Strebel invariant, introduced in \cite{Bierietal1987}; note that we will discuss a slightly restricted version of the invariant, over $\Z$ rather than $\R$. For a group $G$, we will identify the first cohomology group $\operatorname{H}^1(G;\Z)$ with the set of homomorphisms $G \to \Z$. Under this isomorphism, primitive elements of $\operatorname{H}^1(G;\Z)$ correspond to epimorphisms $G \to \Z$. The \emph{BNS invariant} $\Sigma^1(G)$ is the subset of the set of primitive elements of $\operatorname{H}^1(G;\Z)$ consisting of precisely those epimorphisms that are induced by ascending HNN extensions with finitely generated vertex groups \cite[Proposition 3.1 and Theorem 5.2]{Brown1987}.
The BNS invariant satisfies the following property: given an epimorphism $\phi \colon G \to \Z$, the kernel $\ker \phi$ is finitely generated \iff both $\phi$ and $-\phi$ lie in $\Sigma^1(G)$ \cite[Theorem B1]{Bierietal1987}.

\subsection*{Two conjectures}
We close the introduction with two conjectures.  First, recall two related reformulations, one of the Noetherian property and one of coherence of rings: namely a ring $R$ is \emph{Noetherian} if and only if every finitely generated left or right $R$-module is type $\typeFP{\infty}$,
and a ring $R$ is \emph{coherent} if and only if every finitely presented left or right $R$-module is of type $\typeFP{\infty}$.  Put yet another way,
Noetherian rings are the rings where the category of finitely generated (left or right) modules is abelian and coherent rings are the rings where the category of finitely presented (left or right) modules is abelian.

The first conjecture is a long-standing open problem attributed to Baer.  It is well known to hold for elementary amenable groups (we include a proof of this fact in \Cref{Baer for EA}).

\begin{BaersConj}
Let $G$ be a group.  Then $\Z G$ is Noetherian if and only if $G$ is virtually polycyclic.
\end{BaersConj}
Note that one direction of the conjecture is known: that the group ring of a virtually polycyclic group is Noetherian was established by Philip Hall in the nineteen-fifties \cite{Hall1954} and can be considered as a non-commutative version of Hilbert's Basissatz. Hall's result would surely have been an inspiration to Baer and others to begin asking if the converse is true.
As observed by Kropholler--Lorensen~\cite{KrophollerLorensen2019}, the fact that $\Z G$ is Noetherian implies that $G$ must be amenable, and hence in Baer's conjecture we can restrict our attention solely to amenable groups.

We offer the following, analogous conjecture. Let us say that the group ring $\Z G$ is \emph{properly coherent} if it is coherent but not Noetherian.

\begin{conj}
Let $G$ be a finitely generated amenable group.  Then $\Z G$ is properly coherent if and only if $G$ is a properly ascending HNN extension with virtually polycyclic vertex group.
\end{conj}

Note that assumption of amenability is necessary since $\Z F_n$, where $F_n$ denotes the free group of rank $n$, is a semifir, and hence is coherent.  Our \Cref{thm main} verifies the conjecture for elementary amenable groups.

\subsection*{Acknowledgements}
This work has received funding from the European Research Council (ERC) under the European Union's Horizon 2020 research and innovation programme (Grant agreement No. 850930).

The authors are grateful for financial support from the Clay Mathematical Institute `Hot-topics small working groups' scheme which helped facilitate this research.

\section{Radicals}
It will be convenient in the proofs that follow to make use of a variation on the Hirsch--Plotkin radical of a group. Recall that the Fitting subgroup of a group is the join of the nilpotent normal subgroups. Fitting's lemma says that if $H$ and $K$ are nilpotent normal subgroups of a group then $HK$ is also a nilpotent normal subgroup, and consequently the Fitting subgroup is the directed union of the nilpotent normal subgroups. Note that the Fitting subgroup need not necessarily be nilpotent itself, and hence it can be desirable to consider the Hirsch--Plotkin radical of a group. This is defined to be the join of the normal locally nilpotent subgroups and it is a characteristic locally nilpotent subgroup. There is a lemma that underpins the Hirsch--Plotkin radical in the same way that Fitting's lemma underpins the Fitting subgroup: this lemma says that if $H$ and $K$ are normal locally nilpotent subgroups then their join $HK$ is locally nilpotent. The Hirsch--Plotkin radical is the join of all normal locally nilpotent subgroups and the lemma ensures that it is the directed union of the normal locally nilpotent subgroups: in particular the Hirsch--Plotkin radical is the unique largest normal locally nilpotent subgroup.

For our purpose, other radicals are convenient. The following generalizes the Hirsch--Plotkin radical.  If $\mathcal P$ is a subgroup-closed  class of groups then ${\small \mathbf L}\mathcal P$ denotes the class of those groups whose finitely generated subgroups belong to $\mathcal P$.

\begin{lem}\label{General Hirsch Plotkin method}
Let $\mathcal P$ be a subgroup-closed class of finitely generated groups with the property that if $H$ and $K$ are normal $\mathcal P$-subgroups of a group then $HK$ is also a $\mathcal P$-subgroup. Then every group contains a unique largest normal ${\small \mathbf L}\mathcal P$-subgroup.
\end{lem}

\begin{proof}It suffices to prove that the join of any two normal ${\small \mathbf L}\mathcal P$-subgroups again belongs to ${\small \mathbf L}\mathcal P$. Below we adopt the notation that for subsets $X$ and $Y$ of a group, 
$X^Y$ denotes the subgroup generated by the conjugates $y^{-1}xy$, $x\in X,y\in Y$.
Let $H$ and $K$ be normal ${\small \mathbf L}\mathcal P$-subgroups. To prove that $HK$ is an ${\small \mathbf L}\mathcal P$-group, we must show that for finite subsets $X\subseteq H$ and $Y\subseteq K$ the subgroup $\langle X, Y\rangle$ belong to $\mathcal P$. Fix such a choice of $X$ and $Y$, and let $C=\{[x,y]=x^{-1}y^{-1}xy;\ x\in X, y\in Y\}$. Note that $C$ is finite and is a subset of $H\cap K$ because $H$ and $K$ are normal. Therefore $\langle C,X\rangle$ is a finitely generated subgroup of $H$ and so belongs to $\mathcal P$. Since $\mathcal P$ is subgroup closed, we deduce that $C^{\langle X\rangle}$ belongs to $\mathcal P$ and so it is finitely generated. Moreover, $C^{\langle X\rangle}$ is contained in $H\cap K$. Therefore 
$\langle Y,C^{\langle X\rangle}\rangle$ is a finitely generated subgroup of $K$ and hence it belongs to $\mathcal P$. Therefore the subgroup $C^{\langle X\rangle\langle Y\rangle}$ also belongs to $\mathcal P$. Now \cite{RobinsonDJS1996}*{5.1.7} shows that 
$C^{\langle X\rangle\langle Y\rangle}=[\langle X\rangle,\langle Y\rangle]$, the subgroup generated by all the commutators $[z,w]$ with $z \in \langle X \rangle$ and $w \in \langle Y \rangle$, and then we have
$$\langle Y,C^{\langle X\rangle}\rangle=\langle Y,C^{\langle X\rangle\langle Y\rangle}\rangle
=\langle Y, [\langle X\rangle,\langle Y\rangle]\rangle=Y^{\langle X\rangle}.$$ This shows that 
$Y^{\langle X\rangle}$ belongs to $\mathcal P$. By symmetry $X^{\langle Y\rangle}$ belongs to $\mathcal P$. Therefore $\langle X,Y\rangle=Y^{\langle X\rangle}X^{\langle Y\rangle}$ is the join of two normal $\mathcal P$-subgroups of $\langle X,Y\rangle$, and so belongs to $\mathcal P$.
\end{proof}

\begin{cor}
In any group there is a unique largest normal locally polycyclic subgroup and there is a unique largest normal locally \{virtually polycyclic\} subgroup, and both these subgroups are characteristic.
\end{cor}
\begin{proof}
Both the class of polycyclic groups and the class of virtually polycyclic groups are closed under subgroups, quotients, and extensions.  Since virtually polycylic groups are finitely generated, both classes satisfy the hypotheses of \Cref{General Hirsch Plotkin method}.
\end{proof}

\begin{cor}\label{peters adjustment}
Suppose that the group $G$ has a subgroup $H$ of finite index which is an ascending HNN extension over a virtually polycyclic vertex group. Then $G$ itself is virtually polycyclic or is an ascending HNN extension over a virtually polycyclic vertex group.
\end{cor}
\begin{proof}
Since finite-index subgroups of an ascending HNN extension over a virtually polycyclic vertex group are also ascending HNN extensions over virtually polycyclic vertex groups, we may assume that $H$ is normal.

We may suppose that $H=B*_{B,t}$ is an ascending HNN extension where $B$ is a virtually polycyclic subgroup of $H$, where $t$ is an element of $H$ such that $B\subseteq B^t$, and where $K=\bigcup_{n\geqslant 0}B^{t^n}$ is a normal subgroup of $H$ such that $H=K\langle t\rangle$. Being an ascending union of virtually polycyclic subgroups we know that $K$ is a locally virtually polycyclic subgroup and therefore contained in a characteristic such subgroup: namely the radical subgroup $L$ given by \Cref{General Hirsch Plotkin method} and its immediate corollary. Clearly $G/L$ has no non-trivial finite normal subgroups and is virtually cyclic, so $G/L$ is either trivial, infinite cyclic, or infinite dihedral.
When $G = L$, then $G$ is locally virtually polycyclic. Since $H$ is finitely generated, so is $G$, and thus $G$ is virtually polycyclic.

If $G/L$ is infinite dihedral, then $H/(H \cap L) \cong \mathbb Z$, and the two epimorphisms $H \to \mathbb Z$ with kernel $K$ thought of as elements of the first cohomology group of $H$ over $\mathbb Z$ lie in the same orbit under the action of $G$. Since $H$ is an ascending HNN extension with a finitely generated vertex group, one of these cohomology classes lies in the Bieri--Neumann--Strebel invariant $\Sigma^1(H)$; since the two classes are related via an automorphism of $H$, so does the other class. Therefore, the kernel of these two epimorphisms is  
finitely generated. This forces $K= B$, and so $H$ is virtually polycyclic. It follows that $G$ is virtually polycyclic as well.

Finally, if $G/L \cong \mathbb Z$ then using the fact that $G$ is finitely presented and \cite
{BieriStrebel1978}*{Theorem~A} we conclude that $G$ is an HNN extension with vertex group being a finitely generated subgroup of $L$. Since $G$ does not contain a non-abelian free group, the HNN extension is ascending. Also, finitely generated subgroups of $L$ are virtually polycyclic, and we are done.
\end{proof}

\section{Indicability, coherence, and cascading groups}

\begin{dfn}\label{defn indicable}
	A group is \emph{indicable} if it is trivial or maps onto $\Z$.
	
	A group $G$ is \emph{indicably coherent} if for every finitely generated subgroup $H \leqslant G$ and every epimorphism $\phi \colon H \to \Z$, the group $H$ splits as an HNN extension with finitely generated vertex group $A$ in such a way that $\phi$ coincides with the quotient map dividing $H$ by the normal closure of $A$; in such a situation, we will say that the HNN extension \emph{realises}  the epimorphism $\phi$.
\end{dfn}

Note that when $G$ is indicably coherent and does not contain non-abelian free groups, we automatically get that the HNN extension is ascending or descending. Also, homologically coherent groups are indicably coherent, by \cite{BieriStrebel1978}*{Theorem~A}.

Let us now state the crucial property of indicably coherent groups.

\begin{prop}
	If $G$ is indicably coherent and does not contain non-abelian free groups, then every quotient of $G$ is also indicably coherent.
\end{prop}

This result follows immediately from the next lemma.

\begin{lem}
	\label{HNN lemma}
	Let $G$ be an ascending HNN extension with vertex group $A$, inducing an epimorphism $\phi \colon G \to \Z$. If $\rho \colon G \to Q$ is an epimorphism such that $\phi$ factors through $\rho$, then $Q$ is an ascending HNN extension with vertex group $\rho(A)$.
\end{lem}
\begin{proof}
	Let $t \in G$ be the stable letter of the given ascending HNN extension, and let $\iota \colon A \to A$ be the associated monomorphism.
	Observe that $\iota(\ker \rho|_A) = t^{-1} \ker \rho|_A t \leqslant \ker \rho$; it is immediate that $\iota(\ker \rho|_A) \leqslant \iota(A) \leqslant A$, and so
	\[
	\iota(\ker \rho|_A) \leqslant \ker \rho|_A.
	\]
	Therefore, $\iota$ descends to a homomorphism $\iota' \colon \rho(A) \to \rho(A)$. Moreover, if $\iota'(a') = 1$ for some $a' \in \rho(A)$,  then for every $a \in \rho^{-1}(a')$  we have $\iota(a) \in \ker \rho$ and so $a = t \iota (a) t^{-1} \in \ker \rho$ as well, which in turn implies $a' = \rho(a) = 1$. Hence, $\iota'$ is injective.
	
	We may define an ascending HNN extension $H$ with vertex group $\rho(A)$, stable letter $t$, and associated monomorphism $\iota'$. It is clear that $\rho$ induces a quotient map $G \to H$ which can be followed with a quotient map $\sigma \colon H \to Q$ to become $\rho$.
	
	Take $h \in \ker \sigma$. Since $h \in H$, we may write
	\[
	h = t^{n_0} x_1 t^{n_1} \cdots x_m t^{n_m}
	\]
	with $x_i \in \rho(A)$, $n_i \in \Z$.
	Since $\phi$ factors through $\rho$, we must have $\sum n_i = 0$, and so we may rewrite $h = \prod_{i=1}^m t^{-n'_i}x_i t^{n'_i}$.
	 Let $N = \max_i |n'_i|$. We now have $t^{-N} h t^N \in \ker \sigma$ being a product of conjugates of elements of $\rho(A)$ by non-negative powers of $t$. But each such conjugate is obtained from an element of $\rho(A)$ by applying $\iota'$ a suitable number of times, and hence lies in $\rho(A)$ itself. Thus, $h \in \ker \sigma \cap \rho(A) = \{1\}$. We conclude that $\sigma$ is an isomorphism, as desired.
\end{proof}

\begin{dfn}[Elementary amenable hierarchy]
	Define $\mathsf{EA}_0$ to be the class of groups which are finite or abelian.   For each ordinal $\alpha>0$ define $\mathsf{EA}_\alpha$ to be the class of groups $G$ that are of one of the following forms:
	\begin{itemize}
	\item $G=\bigcup_i G_i$, where $G_i\in\mathsf{EA}_{\beta_i}$ with $\beta_i<\alpha$;
	\item and $G=N.Q$, where $.$ is any group extension and $N,Q\in\mathsf{EA}_{\beta}$ with $\beta<\alpha$.
	\end{itemize}
\end{dfn}

A very easy proof by transfinite induction shows that every class $\mathsf{EA}_\alpha$ is closed under taking subgroups. 

Before proceeding, we will need one technical definition.

\begin{dfn}[Cascading groups]
	For an ordinal $\alpha$, we define the class of $\alpha$-cascading groups inductively:
	\begin{itemize}
		\item The class of \emph{$0$-cascading} groups consists of finite groups;
		\item For $\alpha>0$, a group $G$ is \emph{$\alpha$-cascading} if it admits a finite index subgroup $H$ and an epimorphism $H \to \Z$ such that for every splitting of $H$ as an HNN extension realising the epimorphism, if the vertex group is finitely generated, then it is $\beta$-cascading for some $\beta<\alpha$.
	\end{itemize}
A group is \emph{cascading} if it is $\alpha$-cascading for some ordinal $\alpha$.
\end{dfn}

Observe that many groups will satisfy this definition vacuously, since they will have no finitely generated vertex groups at all. The definition really comes to life when combined with indicable coherence.

\begin{lem}
	Finitely generated free-abelian groups are cascading, and the property of being cascading passes to finite index overgroups.
\end{lem}
The proof of the first part is an obvious induction on the rank. The second part is even easier.

\begin{lem}
	\label{cascading lemma}
	Let $\alpha$ and $\beta$ be two given ordinals with $\beta < \alpha$.
	Let $G$ be a finitely generated 
	 group in $\mathsf{EA}_\alpha$, and suppose that $G$ fits into the group extension
	 \[
	 K \to G \xrightarrow{\rho}  Q
	 \]
	 where $K$ lies in $\mathsf{EA}_\beta$. If $Q$ and all finitely generated groups in $\mathsf{EA}_\beta$ are cascading, then so is $G$.
\end{lem}
\begin{proof}
	Suppose that $Q$ is $\gamma$-cascading.
The proof is an induction on $\gamma$. If $\gamma=0$, then $Q$ is finite, and hence $K$ is finitely generated. It lies in $\mathsf{EA}_\beta$, and hence it is cascading, forcing $G$ to be cascading as well.

Now, suppose that $\gamma>0$, and that the result holds for all ordinals strictly smaller than $\gamma$.
Let $R$ be a finite index subgroup of $Q$, and let $\phi \colon R \to \Z$ be an epimorphism such that for every HNN extension decomposition of $R$ realising $\phi$, if the vertex group is finitely generated, then it is $\delta$-cascading for some $\delta<\gamma$. Let $H = \rho^{-1}(R)$, and $\psi = \phi \circ \rho\vert_H \colon H \to \Z$. Observe that $\psi$ is an epimorphism.

Suppose that $H$ splits as an HNN extension realising $\psi$ with finitely generated vertex group $A$. Then, using the fact that this HNN extension must be ascending or descending, \cref{HNN lemma} implies that $R$ splits as an HNN extension realising $\phi$ with vertex group $\rho(A)$. Since $A$ is finitely generated, so is $\rho(A)$, and hence $\rho(A)$ is $\delta$-cascading for some $\delta < \gamma$. But now $A$ fits into the exact sequence
\[
K\cap A \to A \to \rho(A),
\]
and induction tells us that $A$ is $\epsilon_A$-cascading, for some ordinal $\epsilon_A$. It follows that $G$ is $(\bigcup_A \epsilon_A + 1)$-cascading, where the union runs over all finitely generated vertex groups $A$.
\end{proof}

\begin{cor}
	If $G$ is a finitely generated 
	elementary amenable group, then $G$ is cascading.
\end{cor}
\begin{proof}
		Let $G \in \mathsf{EA}_\alpha$ for some ordinal $\alpha$, and suppose that $G$ is finitely generated. 
	The proof is by transfinite induction on $\alpha$.
	
	If $\alpha = 0$, then $G$ is virtually abelian, and hence cascading.
	
	Now suppose that $\alpha>0$, and that the result holds for all ordinals strictly smaller than $\alpha$.
	As $G$ is finitely generated, if it is obtained from groups lower in the hierarchy as a directed union, then it itself lies lower in the hierarchy, and we are done by the inductive hypothesis. Hence we may assume the existence of an epimorphism $\rho \colon G \to Q$ with $Q$ and $K = \ker \rho$ lower in the hierarchy. Since $G$ is finitely generated, so is $Q$, and hence $Q$ is cascading. We may now use \cref{cascading lemma}.
\end{proof}

We will need the following result. 

\begin{lem}
	\label{union of sols}
	A non-empty directed union of soluble groups of derived length $c$ is soluble of derived length $c$.
\end{lem}
\begin{proof}
	Since being soluble of class at most $c$ is defined by a group law, it is a local property, and hence is inherited by directed unions. Our union is  soluble of derived length \emph{equal to} $c$, since it contains a soluble subgroup of class $c$.
\end{proof}

\begin{thm}\label{EA indcoh vSol}
	Let $G$ be a finitely generated elementary amenable group.  If $G$ is indicably coherent, then $G$ is virtually soluble.
\end{thm}
\begin{proof}
	We have already shown that $G$ is $\alpha$-cascading. We now proceed by induction on $\alpha$.
	
	If $\alpha = 0$, then $G$ is finite and we are done. Otherwise, we have a finite-index subgroup $H \leqslant G$ and an epimorphism $  \phi \colon H \to \Z$ from the definition of cascading. Since $G$ is indicably coherent, we find an HNN extension of $H$ realising $\phi$ with finitely generated vertex group $A$. Now, $A$ is $\beta$-cascading for $\beta<\alpha$, and hence, by the inductive hypothesis, it is virtually soluble. Also,  the HNN extension is ascending (without loss of generality). Let $t$ denote its stable letter, and $\iota \colon A \to A$ the associated monomorphism.
	
	Since $A$ is virtually soluble, there exists a finite group $Q$ such that the intersection $B$ of the kernels of all homomorphisms $A \to Q$ is soluble. Moreover, since $A$ is finitely generated, the subgroup $B$ has finite index in $A$. Now, crucially, $\iota$ composed with any homomorphism $A \to Q$ is also such a homomorphism, and so $\iota(B) \leqslant B$. This allows us to construct a new ascending HNN extension, say $K$, with vertex group $B$ and stable letter $t$. It is immediate that $\bigcup_{i \in \N} t^iBt^{-i}$ is a normal subgroup of $\bigcup_{i \in \N} t^iAt^{-i}$, since $B$ is normal in $A$. Moreover, the image of every $t^iAt^{-i}$ in the quotient $\bigcup_{i \in \N} t^iAt^{-i}/\bigcup_{i \in \N} t^iBt^{-i}$ is finite and has cardinality bounded  above by $|A/B|$. Since the quotient is a union of such images, it is finite. Hence $K$ is a finite-index subgroup of $H$. It is also  soluble by \cref{union of sols}.
\end{proof}

\begin{cor}\label{EA coherent then vSol}
	Let $G$ be a finitely generated elementary amenable group. If $G$ is homologically coherent, then $G$ is virtually soluble.
\end{cor}
\begin{proof}
Since $G$ is homologically coherent, for every finitely generated subgroup $H$ of $G$ and every epimorphism $H\to\Z$ we may write $H$ as an  HNN extension with finitely generated (in fact $\typeFP{2}$) vertex group by \cite{BieriStrebel1978}*{Theorem~A}.   Thus, $G$ is indicably coherent.  The result follows from \Cref{EA indcoh vSol}.
\end{proof}

We are now ready to prove the main result.

\medskip

\begin{duplicate}[\Cref{thm main}]
For a finitely generated elementary amenable group $G$ the following are equivalent:
\begin{enumerate}
	\item $G$ is coherent;
		\label{item 1}
	\item $G$ is homologically coherent;
		\label{item 2}
	\item $\Z G$ is coherent;
		\label{item 3}
	\item either $G$ is virtually polycyclic or $G$ is a properly ascending HNN extension with virtually polycyclic vertex group.
	\label{item 4}
\end{enumerate}
\end{duplicate}

\begin{proof}
Clearly, each of \eqref{item 1} and \eqref{item 3} implies \eqref{item 2}. If we assume \eqref{item 4}, then the group $G$ is virtually soluble by \cref{union of sols}, and hence \Cref{Bieri Strebel soluble coherence} shows that $G$ admits a finite index subgroup $H$ with $\Z H$ being coherent.
Coherence of rings can be reformulated as follows: a ring $R$ is coherent if and only if every $R$-homomorphism between finitely generated free $R$-modules has finitely generated kernel. From this point of view it is immediate that $\Z G$ is also coherent, and \eqref{item 2} follows.

 Finally, assume \eqref{item 2}. \Cref{EA coherent then vSol}  implies that $G$ has a  finite-index soluble subgroup which is also homologically coherent.  Items \eqref{item 1} and \eqref{item 3} follow from \Cref{Bieri Strebel soluble coherence}, and the observation that coherence of groups passes to finite-index overgroups, and similarly for group rings. For \eqref{item 4} we additionally need to use \cref{peters adjustment}.
\end{proof}

\section{Noetherian groups}

For context, let us also mention the following result, which is well known but seems to be hard-to-locate in the literature. Note that a group is called \emph{Noetherian} or \emph{slender} when all its subgroups are finitely generated, or equivalently when it satisfies the ascending chain condition $\max$-$s$ on subgroups.

\begin{thm}
	\label{Baer for EA}
	Every Noetherian elementary amenable group $G$ is virtually polycyclic. 
\end{thm}
\begin{proof}
The proof is an induction on the elementary amenable hierarchy. Every finite group is virtually polycyclic, and so is every finitely generated (and hence every Noetherian) abelian group. This shows the result for the class $\mathsf{EA}_0$.

Now suppose that $G$ lies in $\mathsf{EA}_\alpha$, and the result holds for all groups in $\mathsf{EA}_\beta$ for all $\beta<\alpha$. Since $G$ is Noetherian, it is finitely generated. As shown above, this implies that $G$ lies in $\mathsf{EA}_\beta$ for some $\beta<\alpha$, or that we have an epimorphism $\rho \colon G  \to Q$ with $\ker \rho = K$, and both $K$ and $Q$ lie lower in the hierarchy. Since being Noetherian passes to subgroups and quotients, both $K$ and $Q$ are polycyclic-by-finite. By passing to a finite-index subgroup of $G$, we may assume that $Q$ is polycyclic.

Let $N$ be a finite-index normal subgroup of $K$ that is polycyclic. The group $G/N$ is then finite-by-polycyclic, and hence polycyclic-by-finite. We conclude that $G$ is virtually polycyclic.
\end{proof}

\bibliography{ElementaryAmenable}

\end{document}